\documentclass[11pt]{amsart}
\usepackage{amssymb, hyperref}
\usepackage{mathtools}
\usepackage{enumerate}
\numberwithin{equation}{section}
\DeclarePairedDelimiter\abs{\lvert}{\rvert}%
\DeclarePairedDelimiter\norm{\lVert}{\rVert}%
\DeclarePairedDelimiter\bnorm{\biggl\lVert}{\biggr\rVert}%
\makeatletter
\let\oldabs\abs
\def\abs{\@ifstar{\oldabs}{\oldabs*}}
\let\oldnorm\norm
\def\norm{\@ifstar{\oldnorm}{\oldnorm*}}
\makeatother

\theoremstyle{plain}
\newtheorem{thm}{Theorem}[section]
\newtheorem{theorem}{Theorem}
\newtheorem*{assume}{Assumption}

\newtheorem{lem}[thm]{Lemma}
\newtheorem{prop}[thm]{Proposition}

\theoremstyle{definition}

\theoremstyle{remark}

\newcommand\R{\mathbb{R}}
\newcommand\Z{\mathbb{Z}}
\newcommand\N{\mathbb{N}}

\newcommand\D{\mathcal{D}}
\newcommand\K{\mathcal{K}}
\newcommand\E{\mathcal{E}}
\newcommand\Q{\mathcal{Q}}

\newcommand{\ben}{\begin{enumerate}[(i)]}
\newcommand{\een}{\end{enumerate}}
\newcommand{\ft}{\mathcal{F}}

\newcommand{\iftd}{\mathcal{F}^{-1}_{\mathbb{R}^d}}
\newcommand{\wh}{\widehat}
\newcommand{\les}{\lesssim}
\newcommand{\ges}{\gtrsim}
\newcommand{\supp}{\operatorname{supp}}
\newcommand{\meas}{\operatorname{meas}}
\newcommand{\diam}{\operatorname{diam}}
\newcommand{\rad}{\operatorname{rad}}

\newcommand{\inn}[1]{\langle #1 \rangle}
\newcommand{\sumab}[2]{\sum_{\substack{ {#1} \\ {#2} }}}

\begin{document}
\title[Quasiradial Fourier multipliers]{Endpoint bounds for quasiradial Fourier multipliers}
\author{Jongchon Kim}
\address{Department of Mathematics, University of Wisconsin-Madison, Madison, WI 53706 USA}
\email{jkim@math.wisc.edu}
\begin{abstract}
We consider quasiradial Fourier multipliers, i.e. multipliers of the form $m(a(\xi))$ for a class of distance functions $a$. We give a necessary and sufficient condition for the multiplier transformations to be bounded on $L^p$ for a certain range of $p$. In addition, when $m$ is compactly supported in $(0,\infty)$, we give a similar result for associated maximal operators.
\end{abstract}
\maketitle
\section{Introduction}
This paper is concerned with $L^p$ estimates for a class of Fourier multiplier transformations $T_m$ given by
$$\widehat{T_m f} = m\widehat{f}$$
for a function $m$. Many deep results have been obtained for specific multipliers, such as the Bochner-Riesz multipliers (see, e.g., \cite{BG}). However, in the range $1\leq p\leq 2$, a characterization of $m$ for which $T_m$ is bounded on $L^p$ is known only for $p=1,2$ (see \cite{Hor}) and it is generally believed that such a characterization in reasonable terms is impossible when $1<p<2$. 

It came as a surprise that radial Fourier multipliers $m$ for which $T_m$ is bounded on $L^p$ can be characterized by a simple condition on the convolution kernel $\mathcal{F}^{-1} m$. Garrig\'{o}s and Seeger \cite{GS} gave such a characterization when $T_m$ acts on $L^p_{\rad}$, the space of radial $L^p$ functions. A breakthrough by Heo, Nazarov and Seeger \cite{HNS} extended the result to entire $L^p$ spaces, provided that the dimension is sufficiently high. In addition, Lee, Rogers and Seeger \cite{LRS} obtained an endpoint result for the $L^p$ boundedness of $T_m$ and the associated maximal operators 
$$M_m f(x) := \sup_{t>0} |T_{m(\cdot/t)} f(x)|$$
in terms of Besov spaces for a larger $p$-range. Furthermore, a characterization result for $M_m$ was obtained by the author \cite{KimMax}.

The goal of this paper is to generalize the results for radial Fourier multipliers to a class of quasiradial Fourier multipliers, i.e. 
multipliers of the form $m\circ a$, where $m$ is a function on $(0,\infty)$ and $a$ is a smooth positive function on $\R^d\setminus 0$ which is homogeneous of degree 1. In what follows, we let 
$$ \widehat{T_mf}(\xi)  = m(a(\xi)) \widehat{f}(\xi)$$
and $M_m$ be the maximal operator associated with $T_m$. A more general class of quasiradial Fourier multipliers has been also studied; see \cite{SeQ} and references therein.

Note that $\nabla a (\xi) \neq 0$ and the ``cosphere'' $\Sigma = \{ \xi : a(\xi) = 1 \}$ is a
smooth compact hypersurface. Throughout the paper, we shall further assume that 
\begin{assume}
$\Sigma$ has everywhere nonvanishing Gaussian curvature. 
\end{assume}

Let us first consider the case when $m$ is compactly supported in $(0,\infty)$ and discuss a necessary condition for the validity of the inequalities 
$\norm{T_mf}_{L^{p,q}} \les \norm{f}_{L^p}$ and $\norm{M_mf}_{L^{p'}} \les \norm{f}_{L^{p',q'}}$ for $1<p<2$, where $L^{p,q}=L^{p,q}(\R^d)$ is
the Lorentz space and $p'$ is the exponent dual to $p$. We remark that $L^{p,p} =L^p$ and $L^{p,\infty}$ is the weak $L^p$ space and refer the reader to \cite{BL} for a detailed treatment of the real interpolation method to be used in the present paper. By using a Schwartz function whose Fourier transform is 1 on the support of $m\circ a$, one obtains
$$\norm{\iftd[m\circ a]}_{L^{p,q}} \les \norm{T_m}_{L^{p} \to L^{p,q}} \leq \norm{M_m}_{L^{p',q'} \to L^{p'}},$$
where the second inequality is by duality; $\norm{T_m}_{L^{p} \to L^{p,q}} = \norm{T_m}_{L^{p',q'} \to L^{p'}}$.

By a result due to Lee and Seeger \cite[Theorem 2.2]{LS}, we have
\begin{equation}\label{eqn:equivalence}
\norm{\iftd[m\circ a]}_{L^{p,q}} \simeq \norm{\frac{\widehat{m}}{(1+|\cdot|)^{(d-1)/2}}}_{L^{p,q}(\nu_d)},
\end{equation}
where $\nu_d$ is the measure \[ d\nu_d(r) = (1+|r|)^{d-1}dr\] on $\R$, which was shown to hold even without the curvature assumption on $\Sigma$. It would be convenient to work with the norm on the right-hand side, since we shall rely on the Fourier inversion formula
\begin{equation}\label{eqn:inversion}
m(a(\xi)) = \frac{1}{2\pi} \int \widehat{m}(r) e^{ira(\xi)} dr.
\end{equation}

Next, consider the case when $m$ is not necessarily compactly supported, and suppose that $\norm{T_mf}_{L^{p,q}} \les \norm{f}_{L^p}$. Take a nontrivial smooth function $\phi$ which is compactly supported in $(0,\infty)$. Then by testing with $f= \iftd[\phi\circ a](t\cdot)$ for $t>0$, one obtains by \eqref{eqn:equivalence},
$$\sup_{t>0} \norm{\frac{\ft_\R[m(t\cdot)\phi]}{(1+|\cdot|)^{(d-1)/2}}}_{L^{p,q}(\nu_d)} \les \norm{T_m}_{L^p \to L^{p,q}}.$$

With the above notations, we are ready to state our main results.
\begin{thm} \label{thm:B}
Let $d\geq 4$, $1<p<\frac{2(d-1)}{d+1}$, and $p\leq q\leq \infty$. Then 
\begin{equation*}
 \norm{T_m}_{L^{p} \to L^{p,q}} \simeq \sup_{t>0} \norm{\frac{\ft_\R[m(t\cdot)\phi]}{(1+|\cdot|)^{(d-1)/2}}}_{L^{p,q}(\nu_d)} . 
\end{equation*}
\end{thm}
The result for the special case $a(\xi)=|\xi|$ is obtained in \cite{HNS,HNS1}. For compactly supported $m$, we may also characterize the $L^p$ boundedness of $M_m$ in the dual $p$-range.

\begin{thm} \label{thm:thmBMax}
Let $d\geq 4$, $1<p<\frac{2(d-1)}{d+1}$, and $p\leq q\leq \infty$. Assume that $m$ is compactly supported in $(0,\infty)$. Then
\begin{equation*}
 \norm{M_m}_{L^{p',q'} \to L^{p'}} \simeq \norm{T_m}_{L^{p} \to L^{p,q}} \simeq  \norm{\frac{\widehat{m}}{(1+|\cdot|)^{(d-1)/2}}}_{L^{p,q}(\nu_d)}.
\end{equation*}
\end{thm}

For the proof of Theorem \ref{thm:B}, it is natural to apply ideas from \cite{HNS,HNS1} for the radial case $a(\xi)=|\xi|$. However, we need a different formulation by the use of \eqref{eqn:inversion} to handle more general distance functions $a$, which seems to be at least originated from the study of the Weyl formula (see \cite{HorSpec} for a discussion). A new difficulty is the existence of Schwartz tails, which demands finer estimates relying on both the decay and the oscillation of the Fourier transform of a smooth measure on $\Sigma$.

There is also a generalization of the radial Fourier multiplier result by Lee, Rogers and Seeger \cite{LRS} to quasiradial Fourier multipliers. For this, we refer the reader to \cite{KimQuasi}. This paper is a shortened version of \cite{KimQuasi} excluding that result, which will appear elsewhere in a more general setting: spectral multipliers of pseudo-differential operators on compact manifolds (see \cite{KimSpec}).

It would be interesting to see what happens if the curvature assumption on $\Sigma$ is relaxed. However, the argument we present here alone does not seem to allow any endpoint result if we do not assume the maximal decay rate of the Fourier transform of the surface measure on $\Sigma$. The situation is quite different in $L^1$. Sharp weak-type (1,1) estimates for the Riesz means $R^\lambda_t$ were obtained by Christ and Sogge \cite{CS} without any curvature assumption on $\Sigma$. For the general $\Sigma$ which is a boundary of a convex body in $\R^2$ containing the origin, we refer the reader to \cite{Cladek} and references therein. 

\subsection*{Structure of the paper} We organize the paper as follows. In Section \ref{sec:B}, we derive a model inequality which implies a version of Theorem \ref{thm:B} for multipliers compactly supported in $(0,\infty)$. Then the inequality is reduced to a restricted weak type inequality, the proof of which is completed by a crucial $L^2$ estimate in Section \ref{sec:L2}. In Section \ref{sec:proofmaximal}, we derive the Lorentz space version of the model inequality for the compactly supported multipliers and then prove Theorem \ref{thm:thmBMax}. In Section \ref{sec:proofmultiplier}, we prove Theorem \ref{thm:B}. 

\subsection*{Notation} We let $N$ be a sufficiently large natural number (with respect to $d$) which may differ from line to line. We write $A \les B$ or $A=O(B)$ if $|A| \leq C B$ for an absolute constant $C>0$ which may depend on parameters such as $a, N, \epsilon, d, p, q$. We write $A \simeq B$ if $A \les B \les A$. We shall often abbreviate $\norm{f}_{L^p(\R^d)}$ to $\norm{f}_p$ and omit multiplicative constants which depend only on $d$, such as $(2\pi)^{-d}$. Throughout the paper, $p$ denotes a real number $1<p<2$ unless otherwise stated.

\subsection*{Acknowledgments} This paper will be a part of the author's PhD thesis. He would like to thank his advisor Andreas Seeger for suggesting this problem and the use of \eqref{eqn:inversion}, and helpful discussions. He also thank a referee for carefully reading the manuscript and pointing out several misprints. This work was supported in part by the National Science Foundation. 

\section{A model inequality} \label{sec:B}
\subsection*{Introduction} \label{sec:introc}
We first consider the case when $m$ is supported in $[1/2,2]$. Let $\eta$ be a smooth function supported in $[1/8,8]$ and $\eta = 1$ on
$[1/4,4]$, so that $m(\cdot /t)\eta = m(\cdot/t)$ if $t\in [1/2,2]$.

By \eqref{eqn:inversion}, we may write 
\begin{align*}
T_m g(x) &= \iint m(a(\xi)) \eta(a(\xi)) e^{i(x-y) \cdot \xi} g(y) dy d\xi \\
&= \iiint \eta(a(\xi))e^{ira(\xi)} e^{i(x-y) \cdot \xi } d\xi \widehat{m}(r)  g(y)dy dr \\
&= \int K_r * f(r, x) dr, 
\end{align*}
where $K_r$ is defined by $\wh{K_r}(\xi) = \eta(a(\xi)) e^{ir a(\xi)}$ and $f(r,y) = \widehat{m}(r) g(y)$.  Here, we have used the convention which shall be used throughout the paper; 
$$ K_r * f (r,x) = \int K_r(x-y) f(r,y) dy. $$

By the above discussion, in order to obtain
$$ \norm{T_mg}_{L^p} \les \norm{\frac{\widehat{m}}{(1+|\cdot|)^{(d-1)/2}}}_{L^p(\nu_d)} \norm{g}_{L^{p}},$$
it suffices to prove
\begin{prop} \label{prop:char}
Let $d\geq 4$ and $1\leq p<\frac{2(d-1)}{d+1}$. Then
\begin{equation} \label{eqn:goalc1}
\bnorm{ \int_\R  K_r *  f  (r,\cdot) dr }_p^p \les 
\int_{\R^d} \int_\R \abs{\frac{f(r,y)}{(1+|r|)^{(d-1)/2}}}^p d\nu_d(r) dy .
\end{equation}
\end{prop}

As a corollary of Proposition \ref{prop:char}, by duality, an endpoint version of the local smoothing estimate for half-wave operators $e^{ita(D)}$ for the dual range $p' > 2 + \frac{4}{d-3}$ can be obtained, strengthening a result by Heo and Yang \cite{HY}. However, this is not new and more general results were obtained by Lee and Seeger \cite{LS1}.

It is often convenient to work with a rescaled version of \eqref{eqn:goalc1} for the purpose of interpolation; 
\begin{equation} \label{eqn:goalc1rescale}
\bnorm{ \int_\R  K_r *  f  (r,\cdot) (1+|r|)^{(d-1)/2} dr }_p \les 
\norm{f}_{L^p(\R\times \R^{d}, \nu_d \times m_d)},
\end{equation}
where $m_d$ denotes the Lebesgue measure on $\R^d$.

In this section, we start the proof of Proposition \ref{prop:char}, which will be completed in Section \ref{sec:L2}.

\subsection{Preliminary estimates} \label{sec:kernel}
In this section, we recall bounds for the kernel $K_r$. One has
the $L^1$ estimate
\begin{equation}\label{eqn:l1}
\norm{K_r}_1 \les (1+|r|)^{\frac{d-1}{2}}. 
\end{equation}
This is a rescaled version of the estimate in \cite{SSS} which continues to hold 
without any curvature assumption on $\Sigma$. 

On the other hand, given the curvature assumption, one may obtain a
rather precise pointwise estimate of the kernel by the method of stationary
phase. For each $x \in \R^d \setminus 0$, there are exactly two points
$\xi_\pm(x)$ on $\Sigma$ which have their outward unit normal vectors $\pm
x/|x|$. Define $\psi_\pm(x) = \inn{x, \xi_\pm(x)}$. Note
that $\psi_\pm(x) = \pm|x|$ in the model case $a(\xi)
= |\xi|$. For a later reference, we note
that $\psi_\pm$ is smooth and homogeneous of degree $1$, $\psi_+ > 0 $, $\psi_-
<0$, and
$c_1 |x| \leq |\psi_\pm(x)| \leq c_2|x|$ for some constants $c_1,c_2>0$ which
depend only on $a$. By the method of stationary phase, we have
\begin{equation} \label{eqn:stationary}
\int_\Sigma e^{ix\cdot \xi'} d\mu(\xi') = \sum_{\pm} e^{i\psi_\pm(x)}b_\pm(x),
\end{equation}
where $b_+$ and $b_-$ are symbols of order $-(d-1)/2$ (see \cite{Sogge}). 

There is a polar coordinate with respect to $\Sigma$, namely $\xi = \rho \xi'$ where $\rho = a(\xi)$ and $\xi' \in \Sigma$. Let $n(\xi')$ be the outward unit normal vector at $\xi' \in \Sigma$ and $d\sigma$ be the surface measure on $\Sigma$. Then the Lebesgue measure on $\R^d$ admits the representation 
\begin{equation} \label{eqn:polar}
d\xi = \rho^{d-1}d\rho d\mu(\xi'),
\end{equation}
where $d\mu(\xi') = \inn{\xi,n(\xi')} d\sigma(\xi')$. We refer the reader to \cite{Dappa} and references therein.

By using the (generalized) polar coordinate \eqref{eqn:polar} and integration by parts exploiting the oscillation $e^{i \psi_\pm(x)}$, the following estimate was obtained in \cite{HY}.
\begin{equation*}
|K_r(x)| \leq C_N (1+|x|+|r|)^{-\frac{d-1}{2}} [ \sum_\pm
(1+|\psi_\pm(x) + r|)^{-N}].
\end{equation*}

Assume that $r\geq 2$ and let $\psi(x) = -\psi_{-}(x)$. Then 
\begin{equation} \label{eqn:ker}
|K_r(x)| \leq C_N (1+|x|+|r|)^{-\frac{d-1}{2}} (1+|\psi(x)-r|)^{-N},
\end{equation}
since $|\psi_-(x) + r |\leq C \psi_+(x) + r$ for some constant $C>0$.

We see that the kernel $K_r$ decays rapidly away from
the set $\{x: |\psi(x) -r | \leq 1\} = r\{x: |\psi(x)-1| \leq 1/r \}$. We remark that the set $\{x: |\psi(x)-1| \leq 1/r \}$ is contained in a $O(1/r)$ neighbourhood of the smooth hypersurface $\{x: \psi(x)=1\}$ and the measure of the set is $O(1/r)$. This gives an alternative way to obtain the $L^1$ estimate \eqref{eqn:l1}.

\subsection{Reduction to a restricted weak type inequality} \label{sec:redc}
We may assume, without the loss of generality, that in \eqref{eqn:goalc1}, the $r$ integration is taken over $r\geq 0$ since the case $r\leq 0$ can be handled similarly. Furthermore, we may assume that $r\geq 2$ by an $L^1$ estimate from \eqref{eqn:l1} (see, e.g., \cite{KimMax}).

We begin with the discretization of the $r$-variable as in \cite{HNS,Seeger}. Assume that we have 
\begin{equation} \label{eqn:disc2}
\bnorm{ \sum_{n\geq 2}  K_{n+u} *  f  (n+u,\cdot) }_p^p \les 
 \int \sum_{n\geq 2}|f(n+u,y)|^p (1+n+u)^{p(d-1)(\frac{1}{p}-\frac{1}{2})} dy,
\end{equation}
with an implicit constant uniform in $0 \leq u< 1$. To prove \eqref{eqn:goalc1}, we write $r=n+u$, where $n\in \N$ and $u\in [0,1)$. Then by Minkowski's inequality and H\"{o}lder's inequality,
\begin{align*}
\norm{\int_{2}^\infty  K_r *  f  (r,\cdot) dr }_p &\leq 
\int_{0}^{1} \bnorm{\sum_{n\geq 2}K_{n+u} *  f  (n+u,\cdot) }_p du \\
&\leq \left(\int_{0}^{1} \bnorm{\sum_{n\geq 2}K_{n+u} *  f  (n+u,\cdot) }_p^p du\right)^{1/p} \\
&\les \left(\int_{2}^\infty \int_{\R^d} \abs{\frac{f(r,y)}{(1+|r|)^{(d-1)/2}}}^p dy d\nu_d(r)\right)^{1/p}.
\end{align*}

Assume that for functions $f$ on $\N\times \R^d$, we have
\begin{equation} \label{eqn:disc}
\bnorm{ \sum_{n\geq 2}  K_{n+u} *  f  (n,\cdot) }_p \les 
\sum_{n\geq 2} \int |f(n,y)|^p dy (1+n)^{p(d-1)(\frac{1}{p}-\frac{1}{2})},
\end{equation}
with an uniform implicit constant for all $0\leq u < 1$. Then we observe that \eqref{eqn:disc2} holds. Thus, we have reduced \eqref{eqn:goalc1} to \eqref{eqn:disc}. 

Next, we discretize the $y$ variable in \eqref{eqn:disc} as in \cite{LRS}. Here, we will assume that $u=0$ for the sake of simplicity, but the argument would clearly indicate that estimates continue to hold with uniform implicit constants provided that $u=O(1)$. For each $(n,z) \in \N\times \Z^{d}$, assume that $b_{n,z}(y)$ is a function normalized by the condition $|b_{n,z}(y)| \leq \chi_{q_{z}}(y)$, where $\chi_{q_{z}}$ is the characteristic function on $q_{z} = \prod_{i=1}^d [z_i,z_i+1)$. Suppose that for any function $\gamma$ on $\N\times \Z^{d}$, we have
\begin{equation} \label{eqn:goalc12}
\bnorm{\sum_{n\geq 2} \sum_{z \in \Z^{d}} \gamma_{n,z}  K_n *  b_{n,z} }_p^p \les
\sum_{n\geq 2} \sum_{z \in \Z^{d}}|\gamma_{n,z}|^p (1+n)^{p(d-1)(\frac{1}{p}-\frac{1}{2})},
\end{equation}
where the implicit constant is independent of the choice of $b_{n,z}$. 

We claim that \eqref{eqn:goalc12} implies \eqref{eqn:disc} (with $u=0$). Note that there is a Schwartz function $\zeta$, such that $K_r * \zeta = K_r$ for any $r\in \R$ by the compact support of $\widehat{K_r}$. Set 
\begin{align*}
\gamma_{n,z} &= \sup_{y\in q_{z}} |\zeta * f (n,y)|, \\
b_{n,z} (y)&= \gamma_{n,z}^{-1}\chi_{q_{z}}(y) \zeta*f(n,y), \; \text{ if $\gamma_{n,z} \neq 0$,}
\end{align*}
and $b_{n,z}(y) = 0$ if $\gamma_{n,z} = 0$. The claim follows since $\sum_{z\in \Z^d} \gamma_{n,z} b_{n,z} = \zeta*f(n,\cdot)$ and
$\sum_{z\in \Z^d} |\gamma_{n,z}|^p \les \int |f(n,y)|^p dy, $
which is a consequence of the fact that $|\gamma_{n,z}|
\les u_N * |f| (n,y)$ uniformly in $y\in q_z$, where $u_N(x) =
(1+|x|)^{-N}$. 

Let $\tilde{\E}_j = (I_j \cap \N) \times \Z^d$, where $I_j = [2^j, 2^{j+1})$. Then \eqref{eqn:goalc12} follows from 
\begin{equation} \label{eqn:goalc123}
\bnorm{\sum_{j\geq 1} 2^{j(d-1)/2} \sum_{(n,z) \in \tilde{\E}_j} \gamma_{n,z}  K_n *  b_{n,z} }_p^p \les
\sum_{j\geq 1} 2^{j(d-1)} \sum_{(n,z) \in \tilde{\E}_j} |\gamma_{n,z}|^p.
\end{equation}

Let $\mu_d$ be the measure on $\N \times \Z^d$ defined by 
$$\mu_d(E) = \sum_{j\geq 1} 2^{j(d-1)} \# \{ (n,z)\in E : n\in I_j \}$$
and let $T$ be the operator acting on functions on $\N\times \Z^d$ by
$$ T \gamma = \sum_{j\geq 1} 2^{j(d-1)/2} \sum_{(n,z) \in \tilde{\E}_j} \gamma_{n,z}  K_n *  b_{n,z}. $$
Then \eqref{eqn:goalc123} is equivalent to the estimate
$ \norm{T \gamma}_{L^p(\R^d)} \les \norm{\gamma}_{L^p(\mu_d)}.$

The case $p=1$ is trivial by the triangle inequality and the $L^1$ estimate \eqref{eqn:l1}. Thus, it suffices to prove the following restricted weak type inequality; For $1<p<\frac{2(d-1)}{d+1}$ and $\lambda>0$,
\begin{equation}\label{eqn:goalc4}
 \meas\{x: |\sum_{j\geq 1} 2^{j(d-1)/2} \sum_{(n,z)\in \E_j }  K_n *
b_{n,z} | > \lambda \} \les \lambda^{-p} \sum_{j\geq 1 } 2^{j(d-1)} \# \E_j,
\end{equation}
where $\E_j$ is a finite subset of $\tilde{\E}_j$ and the implicit constant is independent of the choice of $\E_j$ and $b_{n,z}$. Note that by the $L^1$ estimate, \eqref{eqn:goalc4} holds for $p=1$. Thus, we may assume that $\lambda > C$ for a fixed large constant $C>1$.

\subsection{Density decomposition of $\E_j$}
For each $1\leq k \leq j$, let $D^j_k$ be the collection of (half open) dyadic cubes of side length $2^k$ in $\R^{1+d}$ which tile $[2^j,2^{j+1}) \times \R^d$, and then let $\D_j = \bigcup_{k=1}^j D^j_k$. 

Let $\Q_j(\lambda)$ be the collection of cubes $Q \in \D_j$ such that $$ \# \E_j \cap Q > \lambda^p l(Q),$$
where $l(Q)$ denotes the side length of $Q$. Then let $Q_j(\lambda)$ be the collection of the maximal cubes in $\Q_j(\lambda)$. Note that $Q_j(\lambda)$ is a collection of finitely many disjoint dyadic cubes $Q\in \D_j$. Then we set $$\E_j(\lambda) = \E_j \setminus \bigcup_{Q \in Q_j(\lambda)} Q.$$

This is a dyadic version of the density decomposition (also known as the modified Calder\'{o}n-Zygmund decomposition) given in \cite{HNS}. By the construction, we immediately obtain
\begin{equation} \label{eqn:sumofsidelength}
\sum_{Q\in Q_j(\lambda)}  l(Q)  < \lambda^{-p}\sum_{Q\in Q_j(\lambda)}  \# \E_j \cap Q  \leq \lambda^{-p} \# \E_j,
\end{equation}
and that \begin{equation} \label{eqn:goodpart}
\# \E_j(\lambda) \cap Q \leq \lambda^p l(Q),
\end{equation}
for any $Q\in \D_j$. 

With the notation $B_{n,z} =  K_n *b_{n,z},$ \eqref{eqn:goalc4} follows from
\begin{align} \label{eqn:goalc4high}
 \meas\{x: |\sum_{j\geq 1} 2^{j(d-1)/2} \sum_{Q\in Q_j(\lambda)} \sum_{(n,z)\in \E_j \cap Q }  B_{n,z} | > \lambda \} &\les \lambda^{-p} \sum_{j\geq 1 } 2^{j(d-1)} \# \E_j, \\ \label{eqn:goalc4low}
 \meas\{x: |\sum_{j\geq 1} 2^{j(d-1)/2} \sum_{(n,z)\in \E_j(\lambda) }  B_{n,z}| > \lambda \} &\les \lambda^{-p} \sum_{j\geq 1 } 2^{j(d-1)} \# \E_j.
 \end{align}

In order to prove \eqref{eqn:goalc4low} for $1<p<\frac{2(d-1)}{d+1}$, it suffices to prove the following $L^2$ estimate by Chebyshev's inequality.
\begin{lem} \label{lem:mainc}
Let $d\geq 4$. Assume that $\# \E_j\cap Q \leq \lambda^p l(Q)$ for any $Q\in \D_j$. Then
\begin{equation*}
\bnorm{\sum_{j\geq 1} 2^{j(d-1)/2} \sum_{(n,z)\in \E_j }  B_{n,z}}_{2}^2 \les 
\lambda^{2p/(d-1)} \log_2\lambda \sum_{j\geq 1 } 2^{j(d-1)} \# \E_j.
\end{equation*}
\end{lem}
 
We shall prove Lemma \ref{lem:mainc} in Section \ref{sec:L2}. We shall prove \eqref{eqn:goalc4high} in the next subsection. 

\subsection{The high density part}
Let $(r_Q,y_Q)$ be the center of $Q\in \D_j$ and 
$$\Psi(Q) = \{x: |\psi(x-y_Q) - r_Q | \leq C_1 l(Q) \}, $$
where $C_1 = 5d (\norm{\nabla \psi}_{L^\infty}+1)$. Then $\sum_{(n,z) \in \E_j \cap Q } B_{n,z}$ is essentially supported in $\Psi(Q)$. Indeed, if $(n,z) \in Q$, then 
\begin{equation} \label{eqn:out}
\norm{ B_{n,z} }_{L^1(\Psi(Q)^c)} \les 2^{j(d-1)/2} l(Q)^{-N}.
\end{equation}
To see this, assume that $x \notin \Psi(Q)$ and $y \in q_z$. Then since $|n-r_Q| \leq l(Q)$ and 
$|\psi(x-y)-\psi(x-y_Q)| \leq \norm{\nabla \psi}_{L^\infty} |y-y_Q| \leq C_1 l(Q)/5$, we get 
$$|\psi(x-y)-n| \geq |\psi(x-y_Q) -r_Q| - C_1 l(Q)/2 \geq l(Q).$$ 

Thus we get by the kernel estimate \eqref{eqn:ker}, $$|K_n(x-y)| \les 2^{-j(d-1)/2} l(Q)^{-N} (1+|\psi(x-y)-n|)^{-N}.$$
Since $\int (1+|\psi(x)-n|)^{-N} dx\les 2^{j(d-1)}$ if $n \simeq 2^j$, \eqref{eqn:out} follows by the normalization of $b_{n,z}$.

On the other hand, we have a favourable bound for the measure of $\Psi(Q)$. 
Note that by \eqref{eqn:sumofsidelength},
\begin{align*}
\meas \big[\bigcup_{j\geq 1} \bigcup_{Q \in Q_j(\lambda)} \Psi(Q) \big]
\les  \sum_{j\geq 1} \sum_{Q\in Q_j(\lambda)} 2^{j(d-1)} l(Q) \leq \lambda^{-p} \sum_{j\geq 1} 2^{j(d-1)} \# \E_j.
\end{align*}
Thus, it suffices to prove \eqref{eqn:goalc4high} with $B_{n,z}$ replaced by $B_{n,z}  \chi_{\Psi(Q)^c}$. But then by the $L^1$ estimate \eqref{eqn:out} and Chebyshev's inequality, we may bound the measure by 
\begin{align*}
&\lambda ^{-1} \sum_{j\geq 1} 2^{j(d-1)/2} \sum_{Q\in Q_j(\lambda)}
\sum_{(n,z)\in \E_j \cap Q }  \norm{B_{n,z}}_{L^1( \Psi(Q)^c)} \\
&\les \lambda ^{-1} \sum_{j\geq 1} 2^{j(d-1)} \sum_{Q\in Q_j(\lambda)} l(Q)^{-N}
\#\E_j \cap Q.
\end{align*}
Finally, observe that if $Q\in Q_j(\lambda)$, then $l(Q) \geq \lambda^{p/d}$. This follows from the fact that $\# \E_j \cap Q$, which is at least $\lambda^p l(Q)$, is at most ${l(Q)}^{d+1}$ due to the lattice structure of $\E_j$. If we take $N$ large enough, we obtain \eqref{eqn:goalc4high}.

\section{Proof of Lemma \ref{lem:mainc}}\label{sec:L2}
We may assume that the sum $\sum_{j\geq 1}$ is taken over a $10$-separated set of natural numbers, by breaking the original sum into finitely many sums. Following \cite{HNS}, we shall further break the sum as 
$$ \sum_{j \leq 4 \log_2 \lambda} 2^{j(d-1)/2}  G_j+ \sum_{j > 4\log_2 \lambda}
2^{j(d-1)/2} G_j, $$
where $G_j = \sum_{(n,z)\in \E_j} B_{n,z}.$ Then we obtain
\begin{align*}
\bnorm{\sum_{j\geq 1}  2^{j(d-1)/2} G_j}_2^2 &\les \log_2 \lambda \sum_{j \leq
4\log_2 \lambda} 2^{j(d-1)} \norm{G_j}_2^2 + \bnorm{\sum_{j > 4\log_2 \lambda}
2^{j(d-1)/2} G_j }_2^2 \\
&\les \log_2 \lambda \sum_{j\geq 1} 2^{j(d-1)}\norm{G_j}_2^2 \\
&+ \sum_{4\log_2
\lambda<k<j} 2^{(j+k)(d-1)/2} |\inn{G_j,G_k}|.
\end{align*}

We remark that if $B$ be any ball in $\R^{1+d}$ of diameter $\diam(B)$, then it follows from the assumption that
\begin{equation} \label{eqn:ball}
\# \E_j \cap B \leq C_d \left(1+\frac{\diam(B)}{2^j} \right)^d \lambda^p \diam(B).
\end{equation}
The estimate will be used only when $\diam(B) \les 2^j$.

\subsection{Scalar products}\label{sec:scalar}
Fix a large constant $C>0$. For each $(n,z) \in \E_j$, we decompose $\E_k$, when $k< j$, as follows.
\begin{align*}
\E_k(z) &= \{(n',z')\in \E_k : 2^{j-5} \leq \psi(z'-z) \leq 2^{j+5} \} \\
\E_k^l (n,z) &=
\{(n',z')\in \E_k(z) : |\psi(z'-z) - n| < C 2^k \} &\text{if } l=0, \\
&=\{(n',z')\in \E_k(z) : C 2^{k+l-1} \leq |\psi(z'-z) - n| <C 2^{k+l} \} 
&\text{if }
l \geq 1.
\end{align*}

Note that the ``cylinder" $\{(r,y)\in \R^{1+d}:|\psi(y-z)-n| \leq C 2^k, r\in I_k \}$, thus $\E_k^0(n,z)$, can be covered by  $O(2^{(j-k)(d-1)})$ many balls of radius $2^k$ in $\R^{1+d}$. Similarly, $\E_k^l(n,z)$ can be covered by $O(2^{(j-k)(d-1)}2^l)$ many balls of radius $2^k$. By \eqref{eqn:ball}, we get
\begin{equation} \label{eqn:countingballs}
\# \E_{k}^{l}(n,z) \les \lambda^p 2^{(j-k)(d-1)} 2^{k+l}, 
\end{equation}
for $l \geq 0$, where the implicit constant is independent of $(n,z)\in \E_j$.

Then we may bound $|\inn{G_j,G_k}|$ by I $+$ II $+$ III, where
\begin{align*}
 \text{I} &= \sum_{(n,z)\in \E_j} | \inn{B_{n,z}, \sum_{(n',z')\in \E_k \setminus \E_k(z)} B_{n',z'} }| \\
 \text{II} &= \sum_{(n,z)\in \E_j} \sum_{l\geq 1}
\sum_{(n',z')\in \E_{k}^{l}(n,z)}  |\inn{B_{n,z}, B_{n',z'}} | \\
 \text{III} &= \sum_{(n,z)\in \E_j} \sum_{(n',z')\in
\E_{k}^{0}(n,z)} |\inn{B_{n,z}, B_{n',z'}} |.
\end{align*}
Here, III is the main term.

\begin{lem} \label{lem:scalarc1}
For any pairs $(n,z), (n',z')$, we have
\begin{equation}\label{eqn:innerc0}
|\inn{B_{n,z}, B_{n',z'}}| \les (1+|(n,z)-(n',z')|)^{-(d-1)/2}.
\end{equation}

If $(n',z') \in \E_{k}^{l}(n,z)$ for some $l\geq 1$, then
\begin{equation} \label{eqn:innerc3}
 |\inn{B_{n,z}, B_{n',z'}}| \les 2^{-j(d-1)/2} 2^{-(k+l)N}. 
\end{equation}
\end{lem}

Moreover,
\begin{equation} \label{eqn:innerc1}
| \inn{B_{n,z}, \sum_{(n',z')\in \E_k \setminus \E_k(z)} B_{n',z'} }|
\les 2^{-jN},
\end{equation}
where the implicit constant is independent of $(n,z)$.

\begin{proof}
By Plancherel's theorem, $\inn{B_{n,z}, B_{n',z'}}$ is a constant times 
\begin{align} \label{eqn:scalarcb}
 \iint \tilde{K}(n-n',y'-y) b_{n,z}(y) \overline{b_{n',z'} (y')}
dy dy',
\end{align}
where $\tilde{K}(r,x) = \int |\eta(a(\xi))|^2 e^{ira(\xi)}
e^{ix\cdot\xi} d\xi.$ 

Then we have the following bounds for $\tilde{K} = \tilde{K}(n-n',y'-y)$ by \eqref{eqn:ker}, 
\begin{equation}\label{eqn:boundI}
|\tilde{K}| \les (1+|y'-y|+|n-n'|)^{-(d-1)/2} (1+|\psi(y'-y)-(n-n')|^2)^{-N}.
\end{equation}
By inserting \eqref{eqn:boundI} to \eqref{eqn:scalarcb}, we obtain \eqref{eqn:innerc0}. 
 
Next, observe that
$$|\psi(y'-y)-(n-n')| = |\psi(z'-z)-n| + O(2^k).$$
Thus, if $|\psi(z'-z)-n| \geq C 2^{k+l-1}$ for a sufficiently large constant $C>0$, i.e. $(n',z') \in \E_{k}^{l}(n,z)$, then $|\psi(y'-y)-(n-n')| \ges 2^{k+l}$ and we obtain the additional decay $2^{-(k+l)N}$, which implies \eqref{eqn:innerc3}.

For \eqref{eqn:innerc1}, we observe that if $\psi(y'-y) > 2^{j+5}$ or $\psi(y'-y) < 2^{j-5}$, then 
\begin{equation}  \label{eqn:estIerror}
|\tilde{K} | \leq C_N 2^{-jN} (1+|y'-y|)^{-N},
\end{equation}
since $|n-n'| \simeq 2^j$. By the normalization of $b_{n',z'}$,
$$\sum_{(n',z')\in \E_k \setminus \E_k(z)} |b_{n',z'}|  \leq 2^k.$$
Thus, we may bound $| \inn{B_{n,z}, \sum_{(n',z')\in \E_k \setminus \E_k(z)} B_{n',z'} }|$ by
$$C 2^{-jN} 2^k \iint (1+|y'-y|)^{-N}  \chi_{q_z}(y) 
dy dy' \les 2^{-jN}.$$
\end{proof}

Using Lemma \ref{lem:scalarc1}, we may bound the scalar product $\inn{B_{n,z},B_{n',z'}}$ in the main term III by $2^{-j(d-1)/2}$, since $|z-z'|  \simeq 2^j$. Thus, using \eqref{eqn:countingballs}, we get
$$\text{III} \les \lambda^p 2^{j(d-1)/2} 2^{-k(d-2)} \# \E_j .$$ 
Similarly, we see that $\text{I} \les 2^{-jN} \# \E_j$ and $\text{II} \les \lambda^p 2^{j(d-1)/2} 2^{-kN} \# \E_j $, giving
$$|\inn{G_j,G_k}| \les \lambda^p 2^{j(d-1)/2} 2^{-k(d-2)} \# \E_j .$$ 

Consequently,
\begin{align*}
\sum_{4\log_2
\lambda<k<j} 2^{(j+k)(d-1)/2} |\inn{G_j,G_k}| 
&\les \lambda^p \sum_{k> 4\log_2 \lambda} 2^{-k(d-3)/2} \sum_{j\geq 1} 2^{j(d-1)} \# \E_j \\
&\les \sum_{j\geq 1} 2^{j(d-1)} \# \E_j,
\end{align*}
since $p-2(d-3) \leq 0$ if $d\geq 4$ and $p\leq 2$, which is better than we claimed. We continue to estimate $\sum_{j\geq 1} 2^{j(d-1)}\norm{G_j}_2^2$ in the next subsection.

\subsection{The diagonal part}
We shall decompose $I_j$ as a union of intervals of equal length based on the following lemma.
\begin{lem} \label{lem:subint} Let $I$ be a subinterval of  $I_j$. Set $\E_{j,I} = \{(n,z) \in \E_j : n \in I \}$. Then
$$\bnorm{ \sum_{(n,z)\in \E_{j,I} } B_{n,z} }_{2}^2 \les |I| \# \E_{j,I}.$$
\end{lem}
\begin{proof}
By the Cauchy-Schwarz inequality, 
\begin{align*}
\bnorm{ \sum_{(n,z)\in \E_{j,I} } B_{n,z} }_2^2 &= \bnorm{ \sum_{n\in I} K_n *\big(\sumab{z}{(n,z)\in \E_{j,I} } b_{n,z} \big) }_2^2  \\
&\les |I| \sum_{n\in I} \bnorm{K_n * \big(\sumab{z}{(n,z)\in \E_{j,I} } b_{n,z}  \big)}_2^2 \\
&\les |I| \sum_{n\in I} \bnorm{ \sumab{z}{(n,z)\in \E_{j,I} } b_{n,z}  }_2^2 
 \leq |I| \#\E_{j,I}.
\end{align*}
where we have used Plancherel's theorem, $\norm{\widehat{K_n}}_\infty\les 1$ and the orthogonality of $b_{n,z}$ for each fixed $n$ at the last line.
\end{proof}

Using the lemma with $I=I_j$, we can bound $\norm{G_j}_2^2$ by $C2^j \#\E_{j}$, which is satisfactory for the case $2^j\leq \lambda^{2p/(d-1)}$.

Next, we assume that $2^j \geq \lambda^{2p/(d-1)}$ and decompose $I_j$ as a disjoint union of intervals $\{I\}$ of length $\lambda^{2p/(d-1)}$, except possibly with an interval of length less than $\lambda^{2p/(d-1)}$, which can be easily handled separately. Then we get
\begin{align*}
\norm{G_j}_2^2 \les \sum_I \bnorm{\sum_{(n,z)\in \E_{j,I}} B_{n,z} }_2^2 + \sum_{I} \sum_{(n,z)\in \E_{j,I}}  |\inn{
B_{n,z}, \sumab{I'}{I' \neq I}\sum_{(n',z')\in \E_{j,I'}} B_{n',z'}}|.
\end{align*}
Since Lemma \ref{lem:subint} gives the desired estimate for the diagonal part, we turn to the estimate of the scalar products. 

First, we may discard the contribution of $(n',z')$ with $\psi(z'-z) \geq 2^{j+5}$ in the inner product by an argument similar to the proof of \eqref{eqn:innerc1}. Then the remaining part can be bounded by 
\begin{equation}\label{eqn:diag1}
\sum_{(n,z)\in \E_{j}}  \sumab{(n',z')\in \E_{j}}{|I| \leq |(n',z')-(n,z)| \les 2^j } |\inn{
B_{n,z}, B_{n',z'}}|. 
\end{equation}

By \eqref{eqn:innerc0}, we may bound \eqref{eqn:diag1} by
\begin{align*}
\sum_{(n,z)\in \E_{j}}  \sumab{l\geq 0}{2^l|I| \les 2^j} \sumab{(n',z')\in \E_{j}}{ 2^l|I| \leq |(n',z')-(n,z)| < 2^{l+1}|I|} (1+|(n,z)-(n',z')|)^{-(d-1)/2} \\
\les \sum_{(n,z)\in \E_{j}}  \sum_{l\geq 0} \lambda^p [2^{l}|I|]^{-(d-3)/2}  \les \lambda^p |I|^{-(d-3)/2} \# \E_{j} = \lambda^{2p/(d-1)} \# \E_{j} ,
\end{align*}
where we have used \eqref{eqn:ball} and $d\geq 4$. 

We have shown that $\norm{G_j}_2^2 \les \lambda^{2p/(d-1)} \# \E_j$. This finishes the proof of Lemma \ref{lem:mainc}, which completes the proof of Proposition \ref{prop:char}.

\section{Proof of Theorem \ref{thm:thmBMax}}
\label{sec:proofmaximal}
\subsection{Lorentz space estimates}
By real interpolation, we may extend \eqref{eqn:goalc1rescale} to the Lorentz spaces $L^{p,q}$ for $1\leq q\leq \infty$ and $1<p<2(d-1)/(d+1)$. 
We remark that for $p\leq q \leq \infty$, we have 
$$\norm{f}_{L^{p,q}( \R\times \R^d, \nu_d \times m_d)}^p \les 
\int_{\R^d} \norm{f(\cdot,y)}_{L^{p,q}(\nu_d)}^p dy, $$
when $m_d$ is the Lebesgue measure on $\R^d$. Similarly,
\begin{equation} \label{eqn:lorentz}
\norm{\gamma}_{L^{p,q}( \R\times \Z^d, \nu_d \times m_d)}^p \les 
\sum_{z\in\Z^d} \norm{\gamma(\cdot,z)}_{L^{p,q}(\nu_d)}^p,
\end{equation}
when $m_d$ is the counting measure on $\Z^d$. We refer the reader to \cite[Lemma 9.2]{HNS} for the inequalities.

By rescaling, we obtain the Lorentz space version of Proposition \ref{prop:char}.
\begin{prop} \label{prop:charLorentz}
Let $d\geq 4$ and $1< p<\frac{2(d-1)}{d+1}$ and $p\leq q\leq \infty$. Then
\begin{equation*} 
\bnorm{ \int_\R  K_r * f  (r,\cdot) dr }_{L^{p,q}}^p \les
\int_{\R^d} \norm{\frac{f(\cdot,y)}{(1+|\cdot|)^{(d-1)/2}}}_{L^{p,q}(\nu_d)}^p dy.
\end{equation*}
\end{prop}

\subsection{Proof of Theorem \ref{thm:thmBMax}} \label{sec:proofmax}
By the Littlewood-Paley theory and duality, it suffices to prove that
\begin{equation}\label{eqn:golHNS}
\norm{\int_I T_{m(\cdot/t)} g_t dt}_{L^{p,q}} \les \norm{\frac{\widehat{m}}{(1+|\cdot|)^{(d-1)/2}}}_{L^{p,q}(\nu_d)} \norm{\int_I |g_t| dt}_{L^p},
\end{equation}
with $I=[1,2]$. We refer the reader to \cite{LRS} for details.

Next, we proceed as in \cite{KimMax}, assuming that $m$ is supported in $[1/2,2]$. However, an inspection of the proof would clearly show that the result holds whenever $m$ is compactly supported in $(0,\infty)$. We see (as in Section \ref{sec:introc}) that
$$\int_I T_{m(\cdot/t)} g_t dt = \int K_r * f(r,\cdot) dr,$$ 
where $f(r,y) = \int_I t\widehat{m}(tr) g_t(y) dt$. Observe that 
$$ \norm{\frac{f(\cdot,y)}{(1+|\cdot|)^{(d-1)/2}}}_{L^{p,q}(\nu_d)} \les \norm{\frac{\widehat{m}}{(1+|\cdot|)^{(d-1)/2}}}_{L^{p,q}(\nu_d)} \int_I  |g_t(y)| dt$$
by Minkowski's integral inequality for Lorentz spaces and the change of variable $r\to r/t$. Therefore, Proposition \ref{prop:charLorentz} implies \eqref{eqn:golHNS}. 

\section{Proof of Theorem \ref{thm:B}} \label{sec:proofmultiplier}
\subsection{Independence of the cutoff functions}
In the following subsections, we will prove Theorem \ref{thm:B} for a particular cutoff function. It is justified by the following lemma \cite[Lemma 2.4]{GS}.
\begin{lem} Let $1<p<\infty$ and $1\leq q\leq\infty$. Suppose that $\phi$ and $\eta$ are nontrivial smooth functions with compact supports in $(0,\infty)$. Then
\begin{align}\label{eqn:ind1}
\sup_{t>0} \norm{\frac{ \ft_\R[m(t\cdot)\phi] }{(1+|\cdot|)^{(d-1)/2}}}_{L^{p,q}(\nu_d)} 
\simeq \sup_{t>0}  \norm{\frac{ \ft_\R[m(t\cdot)\eta]}{(1+|\cdot|)^{(d-1)/2}}}_{L^{p,q}(\nu_d)}.
\end{align} 
where the implicit constant is independent of $m$.
\end{lem}

\subsection{Combining Fourier localized pieces}
Let $\phi$ be a smooth cut off function supported in $[1/2,2]$ such that 
$\sum_{k\in \Z} \phi(2^{-k}\rho)^2 = 1$ for $\rho>0$. Let $m_k = m(2^k \cdot)\phi $, $\kappa_k = \widehat{m_k}$,
$\widehat{\K^k}(\xi) = m_k(a(\xi))$, and $\eta_k = 2^{kd} \eta(2^k \cdot)$,
where $\hat{\eta}(\xi) = \phi(a(\xi))$. Then we can write
$$T_mf = \sum_{k\in \Z} \eta_k * T_kf,$$
where 
$T_k f = 2^{kd} \K^k(2^k \cdot) * f.$

With the above notations, we quote a special case of \cite[Theorem 4.2]{LRS}.
\begin{theorem} \label{thm:com} Let $1<p<2$, $p\leq q \leq \infty$, $A>0$, and $C_0\geq 1$. Assume that for every $k\in \Z$ and $l\geq 0$ one can split the kernel into
$$\K^k = \K^{k,sh}_l +\K^{k,lg}_l$$
so that the following properties hold:
\ben
 \item $\K^{k,sh}_l$ is supported in $\{x:|x| \leq C_0 2^{l+1}\}$. 
 \item $\sup_{\xi\in \R^d} \big|\widehat{\K^{k,sh}_l}(\xi)\big| \leq A$.
 \item Let $\{b_z\}_{z\in \Z^d}$ such that $\supp
b_z \subset R^l_z := \prod_{i=1}^d \big[ 2^l z_i, 2^l(z_i+1) \big)$ and
$\norm{b_z}_{L^2} \leq 1$. There is $\epsilon>0$ such that for any
$\gamma \in l^p(\Z^d)$,
\begin{equation} \label{eqn:goalgeneral}
 \bnorm{\K^{k,lg}_l * \big(\sum_z \gamma_z b_z \big)}_{L^{p,q}} \leq A
2^{-l\epsilon}2^{ld(\frac{1}{p}-\frac{1}{2})} \norm{\gamma}_{l^p(\Z^d)}.
\end{equation}
\een
Then the operator $T_m$ extends to a bounded operator from $L^p$ to $L^{p,q}$ with the operator norm not exceeding $C_p A$ for some constant $C_p$.
\end{theorem}

Observe that $$\K^k(x) = \int K_r(x) \kappa_k(r) dr.$$ 
We split $\K^k$ as follows. Fix $l_0\in \N$ such that $2^{l_0} \geq d \sum_{\pm} \norm{\nabla \psi_\pm}_{L^\infty}$ and let $C_0=2^{l_0+1} c_1^{-1}$ where
$c_1$ is the constant such that $|\psi_\pm(x)| \geq c_1|x|$ (See Section
\ref{sec:kernel}). Let $\chi$ be a smooth function supported in
$|x|\leq 2C_0$ such that $\chi(x) = 1$ if $|x|\leq C_0$, and let $\chi_l(x) =
\chi(2^{-l} x)$. For each $l\geq 0$, we split the kernel $\K^k$ into a short
and a long range contribution $$\K^k = \K^{k,sh}_l + \big[\K^{k,lg}_l +
E^{k,lg}_l \big],$$
where
\begin{align*}
 \K^{k,sh}_l(x) &= \chi_l(x) \int_{|r|\leq 2^{l+l_0}} K_r(x)\kappa_k(r)  dr, \\
 \K^{k,lg}_l(x) &= \int_{|r| > 2^{l+l_0}} K_r(x) \kappa_k(r) dr, \\
 E^{k,lg}_l(x) &= [1-\chi_l(x)] \int_{|r|\leq 2^{l+l_0}}  K_r(x) \kappa_k(r) dr.
\end{align*}

Note that $(i)$ in Theorem \ref{thm:com} is immediate from the definition. Let 
\begin{align*}
A = C \sup_{k\in \Z} \norm{\frac{ \kappa_k }{(1+|\cdot|)^{(d-1)/2}}} _{L^{p,q}(\nu_d)}.
\end{align*}
for some large constants $C>0$. 

By Theorem \ref{thm:com}, it suffices to verify, for the proof of Theorem \ref{thm:B}, $(ii)$ and 
\begin{align} \label{eqn:goalgeneralmain}
 \bnorm{\K^{k,lg}_l * \big(\sum_z \gamma_z b_z \big)}_{L^{p,q}} &\leq A
2^{-l\epsilon}2^{ld(\frac{1}{p}-\frac{1}{2})} \norm{\gamma}_{l^p(\Z^d)}\\  \label{eqn:goalgeneralerror}
 \bnorm{E^{k,lg}_l * \big(\sum_z \gamma_z b_z \big)}_{L^{p}} &\leq A
2^{-l\epsilon}2^{ld(\frac{1}{p}-\frac{1}{2})} \norm{\gamma}_{l^p(\Z^d)},
\end{align}
for some $\epsilon>0$ for the range of $p$ given in Theorem \ref{thm:B}. 
In this subsection, we verify $(ii)$ and \eqref{eqn:goalgeneralerror} for $1< p<\frac{2d}{d+1}$. Verification of \eqref{eqn:goalgeneralmain} will be done in the following subsection.

For $(ii)$, observe that 
\begin{align*}
\norm{\widehat{\K^{k,sh}_l}}_{L^\infty} &\les \int |\kappa_k(r)| dr 
= \int \frac{|\kappa_k(r)|}{(1+|r|)^{(d-1)/2}} (1+|r|)^{-(d-1)/2} d\nu_d(r) \\
&\les \norm{\frac{\kappa_k}{(1+|\cdot|)^{(d-1)/2}}}_{L^{p,\infty}(\nu_d)} \norm{(1+|\cdot|)^{-(d-1)/2}}_{L^{p',1}(\nu_d)} \\
&\les \norm{\frac{\kappa_k}{(1+|\cdot|)^{(d-1)/2}}}_{L^{p,\infty}(\nu_d)} \leq A.
\end{align*}
In the last line, we have used that $p<2d/(d+1)$.

For \eqref{eqn:goalgeneralerror}, note that if $|r|\leq 2^{l+l_0}$, then
$$|(1-\chi_l(x)) K_r(x)| \leq C_N 2^{-lN} (1+|x|)^{-N},$$
since $|\psi_\pm(x)+r| \geq |\psi_\pm(x)|/2 \geq 2^{l+l_0}$ if $|x| \geq C_0 2^l$. This implies that 
$$\norm{E^{k,lg}_l}_{L^1}  \les 2^{-lN} \int |\kappa_k(r)| dr.$$

Then we have by the normalization of $b_z$ and the calculation for (ii), 
\begin{align*}
 \bnorm{E^{k,lg}_l * \big(\sum_z \gamma_z b_z \big)}_{L^{p}} &\leq \norm{E^{k,lg}_l}_{L^1} \bnorm{\sum_z \gamma_z b_z }_{L^p} \\
 &\les 2^{-lN} \int |\kappa_k(r)| dr \big( \sum_{z}|\gamma_z|^p\norm{b_z}_{L^p}^p \big)^{1/p} \\
  &\les 2^{-lN} 2^{ld(\frac{1}{p}-\frac{1}{2})} \int |\kappa_k(r)| dr \norm{\gamma}_{l^p(\Z^d)}\\
  &\leq 2^{-lN} 2^{ld(\frac{1}{p}-\frac{1}{2})}  A\norm{\gamma}_{l^p(\Z^d)}.
\end{align*}

\subsection{Proof of \eqref{eqn:goalgeneralmain}}
We need to show that there is $\epsilon>0$ such that
\begin{equation} \label{eqn:goalgeneralmainc}
 \bnorm{\int_{|r|\geq 2^{l+l_0}} K_r*\big(\sum_z \gamma_z b_z \big) \kappa_k(r) dr}_{L^{p,q}} \leq 
A 2^{-l\epsilon} 2^{ld(\frac{1}{p}-\frac{1}{2})} \norm{\gamma}_{l^p(\Z^d)}
\end{equation}
for $1<p<\frac{2(d-1)}{d+1}$ and $p\leq q \leq \infty$.

Let $\mu_d$ be the product measure $\nu_d \times m_d$ on $\R\times \Z^d$, where $m_d$ is the counting measure. Let $S$ be the operator acting on functions on $\R \times \Z^d$ by
$$S \gamma(x) = \sum_{z\in\Z^d} \int_{|r|\geq 2^{l+l_0}} \gamma(r,z) K_r*b_z(x) (1+|r|)^{(d-1)/2} dr.$$ 
Then \eqref{eqn:goalgeneralmainc} is a direct consequence of \eqref{eqn:lorentz} and the following Proposition with the choice of $$\gamma(r,z) = (1+|r|)^{-(d-1)/2} \kappa_k(r) \chi_{|r|\geq 2^{l+l_0}}(r) \gamma_z.$$

\begin{prop} \label{prop:globcpro} Let $d\geq 4$, $1< p<\frac{2(d-1)}{d+1}$ and $1 \leq q \leq \infty$. Then
$$\norm{S\gamma}_{L^{p,q}(\R^d)} \les 2^{-l\epsilon} 2^{ld(\frac{1}{p}-\frac{1}{2})} \norm{ \gamma}_{L^{p,q}(\mu_d)}. $$
\end{prop}

For the proof of Proposition \ref{prop:globcpro}, we need the following lemma.
\begin{lem} \label{lem:L1estimate} Let $b\in L^2(\R^d)$ be a function supported in a cube $Q$ of side length $2^l$. If $|r|\geq 2^{l+l_0}$, then 
$$\norm{K_r*b}_{L^1} \les (1+|r|)^{(d-1)/2} 2^{l/2} \norm{b}_{L^2}.$$
\end{lem}
The proof of Lemma \ref{lem:L1estimate} will be given at the end of this subsection.
\begin{proof}[Proof of Proposition \ref{prop:globcpro}]
For a given $p$, let $p_1$ be $p< p_1<\frac{2(d-1)}{d+1}$.
We shall prove that
\begin{align} \label{eqn:glocp}
\norm{S\gamma}_{L^{p_1}(\R^d)} &\les 2^{ld(\frac{1}{p_1}-\frac{1}{2})} \norm{ \gamma}_{L^{p_1}(\mu_d)}, \\ \label{eqn:gloc1}
\norm{S\gamma}_{L^{1}(\R^d)} &\les 2^{l/2} \norm{ \gamma}_{L^{1}(\mu_d)}.
\end{align}

\eqref{eqn:glocp} is a consequence of \eqref{eqn:goalc1rescale} with
$$f(r,y) = \chi_{\{r:|r|\geq 2^{l+l_0}\}}(r) \sum_{z\in\Z^d} \gamma(r,z) b_z(y).$$

Next, we turn to \eqref{eqn:gloc1}. By Lemma \ref{lem:L1estimate} and the normalization hypothesis, we have
\begin{equation*}\label{eqn:l1e}
\norm{K_r*b_z}_{L^1} \les (1+|r|)^{(d-1)/2} 2^{l/2}.
\end{equation*}
Note that this is an improvement over the trivial estimate by \eqref{eqn:l1}
$$\norm{K_r*b_z}_{L^1} \leq \norm{K_r}_{L^1} \norm{b_z}_{L^1}  \les (1+|r|)^{(d-1)/2} 2^{ld/2}.$$
This is exactly where the $\epsilon$ gain is obtained. By the triangle inequality, 
$$\norm{S\gamma}_{L^{1}(\R^d)} \les  2^{l/2} \sum_{z\in\Z^d} \int_{|r|\geq 2^{l+l_0}} |\gamma(r,z)| (1+|r|)^{d-1} dr = 2^{l/2}  \norm{ \gamma}_{L^{1}(\mu_d)}.$$
We finish the proof by real interpolation of \eqref{eqn:glocp} and \eqref{eqn:gloc1}.
\end{proof}

\begin{proof}[Proof of Lemma \ref{lem:L1estimate}]
Let $c_Q$ be the center of $Q$. As in the proof of Proposition \ref{prop:char}, we may assume that $r\geq 2^{l+l_0}$. Then $K_r*b$ is essentially supported in the $2^l$ neighbourhood of the hypersurface $\{x: \psi(x-c_Q) = r \}$. Indeed, if we set $\Psi_r(Q) = \{ x: |\psi(x-c_Q) - r| \leq 2^{l+l_0} \}$, then we have
$$\norm{K_r*b}_{L^1(\Psi_r(Q)^c)} \leq 2^{-lN} |r|^{(d-1)/2}  \norm{b}_{L^1},$$
which follows from the kernel estimate and 
$$|\psi(x-y) - \psi(x-c_Q)| \leq \norm{\nabla \psi}_{L^\infty} |c_Q-y|
\leq 2^{l+l_0-1}.$$ 

Since $\norm{b}_{L^1} \leq 2^{ld/2} \norm{b}_{L^2}$, we are left with the $L^1$ estimate over $\Psi_r(Q)$. Using that the measure of $\Psi_r(Q)$ is $O(|r|^{d-1} 2^l)$, we have by the Cauchy-Schwarz inequality,
\begin{align*}
\norm{K_r*b}_{L^1(\Psi_r(Q))} \leq |r|^{(d-1)/2} 2^{l/2} \norm{K_r*b}_{L^2} \les
|r|^{(d-1)/2} 2^{l/2} \norm{b}_{L^2}, 
\end{align*}
where we have used Plancherel's theorem and $|\widehat{K_r}(\xi)| \les 1$.
\end{proof}


\begin{thebibliography}{10}

\bibitem{BL}
J.~Bergh and J.~L{\"o}fstr{\"o}m, Interpolation spaces. {A}n introduction,
  Springer-Verlag, Berlin-New York (1976). Grundlehren der Mathematischen
  Wissenschaften, No. 223.

\bibitem{BG}
J.~Bourgain and L.~Guth, \emph{Bounds on oscillatory integral operators based
  on multilinear estimates}, Geom. Funct. Anal. \textbf{21} (2011), no.~6,
  1239--1295.

\bibitem{CS}
F.~M. Christ and C.~D. Sogge, \emph{The weak type {$L\sp 1$} convergence of
  eigenfunction expansions for pseudodifferential operators}, Invent. Math.
  \textbf{94} (1988), no.~2,  421--453.

\bibitem{Cladek}
L.~Cladek, \emph{Multiplier transformations associated to convex domains in
  $\mathbb{R}^2$}, J. Geom. Anal. (2015). doi:10. 1007/s12220-015-9665-8

\bibitem{Dappa}
H.~Dappa, \emph{Quasiradial {F}ourier multipliers}, Studia Math. \textbf{84}
  (1986), no.~1,  1--24.

\bibitem{GS}
G.~Garrig{\'o}s and A.~Seeger, \emph{Characterizations of {H}ankel
  multipliers}, Math. Ann. \textbf{342} (2008), no.~1,  31--68.


\bibitem{HNS}
Y.~Heo, F.~Nazarov, and A.~Seeger, \emph{Radial {F}ourier multipliers in high dimensions}, Acta
  Math. \textbf{206} (2011), no.~1,  55--92.
\bibitem{HNS1}
---{}---{}---, \emph{On radial and conical {F}ourier
  multipliers}, J. Geom. Anal. \textbf{21} (2011), no.~1,  96--117.

\bibitem{HY}
Y.~Heo and C.~W. Yang, \emph{Local smoothing estimates for some half-wave
  operators}, J. Math. Anal. Appl. \textbf{381} (2011), no.~2,  573--581.

\bibitem{Hor}
L.~H{\"o}rmander, \emph{Estimates for translation invariant operators in
  {$L\sp{p}$}\ spaces}, Acta Math. \textbf{104} (1960) 93--140.
\bibitem{HorSpec}
---{}---{}---, \emph{The spectral function of an elliptic operator}, Acta
  Math. \textbf{121} (1968) 193--218.

\bibitem{KimMax}
J.~Kim, \emph{A characterization of maximal operators associated with radial
  {F}ourier multipliers}, arXiv:1501.07669 \hskip -.1cm.

\bibitem{KimQuasi}
---{}---{}---, \emph{Endpoint bounds for quasiradial {F}ourier multipliers},
  arXiv:1511.00019v2 \hskip -.1cm.
  
\bibitem{KimSpec}
---{}---{}---, \emph{Endpoint bounds for a class of spectral multipliers on
  compact manifolds}, Indiana Univ. Math. J., to appear \hskip -.1cm.

\bibitem{LRS}
S.~Lee, K.~M. Rogers, and A.~Seeger, \emph{Square functions and maximal
  operators associated with radial {F}ourier multipliers}, in Advances in
  analysis: the legacy of {E}lias {M}. {S}tein, 273--302, Princeton Univ. Press, Princeton, NJ (2014).

\bibitem{LS1}
S.~Lee and A.~Seeger, \emph{Lebesgue space estimates for a class of {F}ourier
  integral operators associated with wave propagation}, Math. Nachr.
  \textbf{286} (2013), no.~7,  743--755.

\bibitem{LS}
---{}---{}---, \emph{On radial {F}ourier multipliers and almost everywhere
  convergence}, J. Lond. Math. Soc. (2) \textbf{91} (2015), no.~1,  105--126.

\bibitem{SeQ}
A.~Seeger, \emph{On quasiradial {F}ourier multipliers and their maximal
  functions}, J. reine angew. Math. \textbf{370} (1986) 61--73.

\bibitem{Seeger}
---{}---{}---, \emph{Notes on radial and quasiradial {F}ourier multipliers}, in
  Nonlinear Analysis, Function spaces and Applications (NAFSA 10),
  T\v{r}e\v{s}t', Czech Republic (2014).

\bibitem{SSS}
A.~Seeger, C.~D. Sogge, and E.~M. Stein, \emph{Regularity properties of
  {F}ourier integral operators}, Ann. of Math. (2) \textbf{134} (1991), no.~2,
  231--251.

\bibitem{Sogge}
C.~D. Sogge, Fourier integrals in classical analysis, Vol. 105 of
  \emph{Cambridge Tracts in Mathematics}, Cambridge University Press, Cambridge
  (1993).

\end{thebibliography}

\end{document}